\documentclass[11pt]{article}
\usepackage{amsmath}
\usepackage{amsthm}
\usepackage{amssymb}
\usepackage{amscd}
\usepackage{relsize}
\usepackage{enumerate}
\newtheorem{teo}{Theorem}
\newtheorem{prop}{Proposition}
\newtheorem{cor}[teo]{Corollary}
\DeclareMathOperator*{\esup}{ess\,sup}
\DeclareMathOperator*{\osc}{osc}
\DeclareMathOperator*{\einf}{ess\,inf}
\title{{\bf Sampling in unitary invariant subspaces associated to LCA groups}}
\author{
{\bf A.~G. Garc\'{\i}a}\thanks{E-mail:\texttt{agarcia@math.uc3m.es}}, \,\,
{\bf M.~A. Hern\'andez-Medina}\thanks{E-mail:\texttt{miguelangel.hernandez.medina@upm.es}}
{\bf\,\, and \,\, G. P\'erez-Villal\'on}\thanks{E-mail:\texttt{gperez@euitt.upm.es}}
}
\date{}
\begin{document}
\maketitle
\begin{itemize}
\item[*] Departamento de Matem\'aticas, Universidad Carlos III de Madrid,
 Avda. de la Universidad 30, 28911 Legan\'es-Madrid, Spain.
\item[\dag\ddag] Departamento de Matem\'atica Aplicada a las Tecnolog\'{\i}as de la Informaci\'on y las Comunicaciones, E.T.S.I.T., U.P.M.,
 Avda. Complutense 30, 28040 Madrid, Spain.
\end{itemize}
\begin{abstract}
In this paper a sampling theory for unitary invariant subspaces associated to locally compact abelian (LCA) groups is deduced. Working in the LCA group context allows to obtain, in a unified way, sampling results valid for a wide range of problems which are interesting in practice, avoiding also cumbersome notation. Along with LCA groups theory, the involved mathematical technique is that of frame theory which meets matrix analysis when appropriate dual frames are computed.
\end{abstract}

\noindent{\bf AMS}: 22B05; 22D10; 42C15; 94A20.
\section{Introduction}
\label{section1}
This paper gives a general sampling theory  for abstract unitary invariant  subspaces in a Hilbert space $\mathcal{H}$. These subspaces are  constructed by using unitary representations of an LCA group in $\mathcal{H}$. Namely, for a fixed element $a\in \mathcal{H}$ and a unitary representation $h\mapsto U(h)$ of a discrete LCA  group $H$ (maybe a subgroup of a more general LCA group $G$), the invariant subspace $\mathcal{A}_a$ in $\mathcal{H}$ looks like  
\[
\mathcal{A}_a=\Big\{\sum_{h\in H}\alpha_h U(h)a\, :\, \{\alpha_h\}_{h\in H}\in \ell^2(H)\Big\}\,.
\] 
Roughly speaking, we are looking for sampling formulas in $\mathcal{A}_a$ taking into account its unitary invariant structure, i.e., having the form
\[
x=\sum_{j=1}^s \sum_{m \in M}\mathcal{L}_j x(m)\,U(m)c_j\quad \text{in $\mathcal{H}$}\,,
\]
where the generalized samples $\mathcal{L}_j x(m):=\langle x, U(m)b_j\rangle_\mathcal{H}$ are taken at an appropriate subgroup $M$ of $H$ from $s$ fixed elements $b_j\in \mathcal{H}$. The reconstruction process in $\mathcal{A}_a$ is carried out from $s$, not necessarily unique, elements $c_j\in \mathcal{A}_a$. In case the regular unitary representation, i.e.,  that given by the translation operator, the generalized samples are nothing but convolution (average) samples. 

As it was pointed out in \cite{ole:15}, the LCA group approach is not just a unified way of dealing with the four classical  groups $\mathbb{R}, \mathbb{Z}, \mathbb{T}, \mathbb{Z}_N$: signal processing often involves products of these groups which are also LCA groups. For example, multichannel video signal involves the group $\mathbb{Z}^d \times \mathbb{Z}_N$, where $d$ is the number of channels and $N$ the number of pixels of each image.

Besides, the relevance of sampling theory in unitary invariant subspaces of a Hilbert space $\mathcal{H}$  is  a recognized fact nowadays. Indeed, shift-invariant subspaces of $L^2(\mathbb{R})$ \cite{boor1:94,boor2:94,bownik:00,ron:95} or periodic extensions of finite signals \cite{garcia:16,garcia:15} provide remarkable examples where sampling theory plays a fundamental role. See, for instance, \cite{aldroubi:01,aldroubi:02,aldroubi:05,kang:11,sun:03,unser:94,unser:98,zhou:99} and the references therein for sampling results in shift invariant spaces. As a consequence, the availability of an abstract sampling theory for unitary invariant spaces becomes a useful tool to handle these problems in a unified way. Moreover, any notational complication is avoided especially in the multidimensional setting.

As the involved samples will be identified as frame coefficients in a suitable auxiliary Hilbert space (in our case the space $L^2(\widehat{H})$), the relevant mathematical technique is that of frame theory, including the computation of appropriate dual frames  taking care of the unitary invariant structure of the considered subspaces. For harmonic analysis on LCA groups we refer to the classical Refs. \cite{folland:95,hewitt:63,rudin:90}. It is worth to mention also the recent papers \cite{cabrelli:10,ole:15} that we have used throughout the paper. 

The paper is organized as follows: In Section \ref{section2} we include some needed preliminaries and we deduce a suitable expression for the (generalized) samples in terms of some special sequences in $L^2(\widehat{H})$; a complete characterization of  these sequences is proved in Proposition~\ref{caracterizacion}. Section \ref{section3} includes the main sampling result, i.e., Theorem \ref{regular}. In particular, we include an analogous of Kluv\'anek sampling theorem for $H$-shift-invariant subspaces of $L^2(G)$. We conclude the paper with a brief note on what we call the $G$-jitter error.

\section{The mathematical setting}
\label{section2}
Let $G$ be a second countable {\em locally compact abelian} (LCA) Hausdorff group with operation written additively. Let $M< H< G$ be countable (finite or countably infinite) {\em uniform lattices} in $G$. Recall that a  uniform lattice $K$ in $G$ is a discrete subgroup of $G$ such that the quotient group $G/K$ is compact (see, for instance, Ref.~\cite{cabrelli:10}). It is known that if $M< H$ are uniform lattices in $G$ then $H/M$ is a finite group (see \cite[Remark 2.2]{cabrelli:12}).

The dual group of the subgroup $H< G$, that is, the set of continuous characters on $H$ is denoted by $\widehat{H}$. Since $H$ is discrete,  its dual $\widehat{H}$ is compact. We assume that its Haar measure $m_{\widehat{H}}$ is normalized such that $m_{\widehat{H}}(\widehat{H})=1$. The value of the character $\gamma\in\widehat{H}$ at the point $h\in H$ is denoted by $(h,\gamma)\in\mathbb{T}$.  With this normalization of the Haar measure the sequence $\{\chi_h\}_{h\in H}$ defined by
\[
\widehat{H}\ni \gamma \mapsto \chi_h(\gamma)=(h,\gamma)\in \mathbb{T}
\]
turns out to be an orthonormal basis for $L^2(\widehat{H})$ (see, for instance, \cite[Prop.\,4.3]{folland:95}). 

Let $g\in G\mapsto U(g)\in\mathcal{U}(\mathcal{H})$ a {\em continuous unitary representation of $G$} in a complex separable Hilbert space $\mathcal{H}$. Therefore, the mapping $h\in H\mapsto U(h)\in\mathcal{U}(\mathcal{H})$ is a unitary representation of $H$ in $\mathcal{H}$. For a fixed $a\in\mathcal{H}$ let define the $U$-invariant subspace in $\mathcal{H}$
\[
\mathcal{A}_a:=\overline{\operatorname{span}}\big\{U(h)a\,:\, h\in H\big\}\subset \mathcal{H}\,.
\]
We assume that $\{U(h)a\}_{h\in H}$ is a Riesz sequence in $\mathcal{H}$; a necessary and sufficient condition can be found in \cite{barbieri:15, hernandez:10}. Thus, the subspace $\mathcal{A}_a$ can be expressed as
\[
 \mathcal{A}_a=\Big\{  \sum_{h\in H}\alpha_h U(h)a\, :\, \{\alpha_h\}_{h\in H}\in \ell^2(H)\Big\}\subset \mathcal{H}\,.
\]
As usual, $\{\alpha_h\}_{h\in H}\in \ell^2(H)$ means that $\sum_{h\in H} |\alpha_h|^2 <\infty$.
\subsection*{The isomorphism $\mathcal{T}_{H,a}$}
We define the isomorphism $\mathcal{T}_{H,a}$ which maps the orthonormal basis $\{\chi_h\}_{h\in H}$ for $L^2(\widehat{H})$ onto the Riesz basis $\{U(h)a\}_{h\in H}$ for $\mathcal{A}_a$, that is
\begin{equation}
\label{iso}
\begin{array}[c]{ccll}
\mathcal{T}_{H,a}: & L^2(\widehat{H}) & \longrightarrow & \mathcal{A}_a\\
       & \displaystyle{\sum_{h\in H} \alpha_h \chi_h} & \longmapsto & \displaystyle{\sum_{h\in H} \alpha_h U(h)a}
\end{array}
\end{equation}
This isomorphism $\mathcal{T}_{H,a}$ has the following simple but important {\em shifting property} with respect to the unitary representation:
\begin{prop}
For any $F \in L^2(\widehat{H})$ and $k\in H$, we have
\begin{equation}
\label{sp}
    \mathcal{T}_{H,a}(F\chi_k)=U(k)(\mathcal{T}_{H,a} F)
\end{equation}
\end{prop}
\begin{proof}
Let $F=\sum_{h\in H} \alpha_h \chi_h$ in $L^2(\widehat{H})$. Then,
\[
\begin{split}
   \mathcal{T}_{H,a}(F\chi_k)&=\mathcal{T}_{H,a}(\sum_{h\in H}\alpha_h 	\chi_h\chi_k )=\sum_{l\in H} \alpha_{l-k}U(l)a \\
   &=\sum_{h\in H} \alpha_h U(k+h)a=U(k)(\mathcal{T}_{H,a} F)\,,
   \end{split}
\]
taking into account that $\chi_h \chi_k=\chi_{h+k}$ and the commutativity in the subgroup $H$.
\end{proof}
\subsection*{An expression for the generalized samples}
For any fixed $b\in \mathcal{H}$ we define the $U$-system $\mathcal{L}_b$ as a linear operator between $\mathcal{H}$ and $C(G,\mathbb{C})$, the space of continuous functions defined on $G$ taking values in $\mathbb{C}$ and given by
\[
\mathcal{H}\ni x \mapsto \mathcal{L}_bx\in C(G,\mathbb{C})\,\, \text{ such that }\,\, \mathcal{L}_bx(g):=\langle x, U(g)b\rangle_\mathcal{H},\quad g\in G\,.
\]

Suppose that $s$ vectors $b_j\in G$, $j=1,2,\dots,s$, are given and consider their associated $U$-systems denoted as
$\mathcal{L}_j:=\mathcal{L}_{b_j}$, $j=1,2,\dots,s$.

The main goal of this paper is the stable recovery of any $x\in \mathcal{A}_a$ from the sequence of its samples taken at the subgroup $M$, that is
\[
   \mathcal{L}_jx(m)=\langle x, U(m)b_j\rangle_\mathcal{H},\quad m\in M,\,\,\, j=1,2,\dots,s.
\]
First we obtain an alternative expression for the sample $\mathcal{L}_jx(m)$. For $x\in \mathcal{A}_a$, let $F$ be  the element in $L^2(\widehat{H})$ such that
$\mathcal{T}_{H,a}F=x$. For $j=1,2,\dots,s$ and $m\in M$ we have
\[
  \mathcal{L}_jx(m)=\langle x, U(m)b_j\rangle_\mathcal{H}=\Big\langle \sum_{h\in H} \alpha_h U(h)a, U(m)b_j\Big\rangle_\mathcal{H}
  =\sum_{h\in H} \alpha_h \overline{\big\langle U(m-h)b_j, a\big\rangle}_\mathcal{H}\,.
\]
Therefore, for any fixed $m\in M$ we have
\[
\begin{split}
\mathcal{L}_jx(m)&=\Big\langle F, \sum_{h\in H} \langle U(m-h)b_j, a\rangle_\mathcal{H}\, \chi_h\Big\rangle_{L^2(\widehat{H})}\\
&=
\Big\langle F,\Big(\sum_{k\in H} \langle U(k)b_j, a\rangle_\mathcal{H} \,\chi_{-k}\Big)\chi_m\Big\rangle_{L^2(\widehat{H})}\,,
\end{split}
\]
where $k=m-h$ runs over $H$. Hence, we obtain the expression
\begin{equation}
\label{samples}
\mathcal{L}_jx(m)=\big\langle F,\overline{G}_j\,\chi_m\big\rangle_{L^2(\widehat{H})},\quad m\in M,\,\,\, j=1,2,\dots,s\,,
\end{equation}
where the function $G_j\in L^2(\widehat{H})$ is given by
\begin{equation}
\label{g_j}
G_j=\sum_{h\in H} \mathcal{L}_ja(h)\,\chi_h, \quad j=1,2,\dots,s.
\end{equation}
It is worth to mention that one could consider the samples taken at any orbit of the subspace $M$, i.e., $\{\mathcal{L}_jx(g_0+m) \}_{m\in M;\, j=1,2,\dots,s}$, where $g_0$ is a fixed element in $G$.

As a consequence of expression \eqref{samples}, the recovery of any $x\in\mathcal{A}_a$ depends on the frame properties of the sequence
$\big\{\overline{G}_j\,\chi_m\big\}_{m\in M;\, j=1,2,\dots,s}$ in $L^2(\widehat{H})$. 

\subsection*{Frame properties of the system $\big\{\overline{G}_j\,\chi_m\big\}_{m\in M;\, j=1,2,\dots,s}$}

The results in this section remains true for arbitrary functions $G_j$ in $L^2(\widehat{H})$, $j=1, 2, \ldots,s$, not necessarily those given by \eqref{g_j}.
First, we need to introduce some necessary preliminaries. The {\em annihilator} of $M$ in $\widehat{H}$ is the closed subgroup
\[
  M^\bot=\{\gamma\in \widehat{H} \,:\, (m,\gamma)=1 \text{ for all $m$ in $M$ }\}
\]
The annihilator of $M$ in $\widehat{H}$ is a finite subgroup of $\widehat{H}$ since $M^\bot$ is isomorphic to $\widehat{H/M}$, and $H/M$ is finite. Let $r$ be the order of $M^\perp$ and set $M^\bot=\big\{\mu_0^\bot=0,\mu_1^\bot,\dots,\mu_{r-1}^\bot\big\}$.
There exists a measurable (Borel) {\em section $\Omega$} of $\widehat{H}/M^\bot$ (see the seminal Ref.~\cite{feldman:68}), i.e.,  a measurable set  $\Omega$ such that
\[
\widehat{H}= \bigcup_{n=0}^{r-1} (\mu^\bot_{n}+\Omega)\quad \text{and}\quad(\mu^\bot_{n}+\Omega) \, \cap \,(\mu^\bot_{n'}+\Omega)=\emptyset, \quad \text{for}\, \, n\neq n'\,.
\]
Notice that $m_{\widehat{H}}(\Omega)=1/r$.  Besides, the sequence $\{\chi_m\}_{m\in M}$ is an orthogonal basis for $L^2(\Omega)$ (see \cite[Prop. 2.16]{cabrelli:10}). 

Our analysis is based in the following expression:  For $F, G_j\in L^2(\widehat{H})$, $j=1,2,\ldots,s$, we have
\begin{equation}
\label{ole}
\big\langle F,\overline{G}_j\,\chi_m\big\rangle_{L^2(\widehat{H})}=\int_{\Omega} \mathbf{G}_j^\top(\xi)\mathbf{F}(\xi)
\overline{(m,\xi)} \,dm_{\widehat{H}}(\xi)=(\mathbf{1}_{\Omega} \mathbf{G}_j^\top\mathbf{F})^{\wedge}(m)\,,\quad m\in M
\end{equation}
where  $\mathbf{1}_{\Omega}$ denotes the characteristic function of $\Omega$, the hat denotes the Fourier transform on the group $\widehat{H}$, i.e., $f^{\wedge}(m)=\int_{\widehat{H}} f(\xi) \overline{(m,\xi)}\,dm_{\widehat{H}}(\xi)$ for any $f\in L^1(\widehat{H})\cap L^2(\widehat{H})$, and
\[
\begin{split}
\mathbf{F}(\xi)&=\big[F(\xi+\mu_0^\bot), F(\xi+\mu_1^\bot), \dots, F(\xi+\mu_{r-1}^\bot) \big]^\top,\\
\mathbf{G}_j(\xi)&=\big[G_j(\xi+\mu_0^\bot), G_j(\xi+\mu_1^\bot), \dots, G_j(\xi+\mu_{r-1}^\bot) \big]^\top, \quad \xi \in \Omega\,.
\end{split}
\]
Indeed, using that $\widehat{H}=\bigcup_{n=0}^{r-1} \big(\mu^\perp_{n}+\Omega \big)$, for any $m\in M$ we obtain  
\[
\begin{split}
\big\langle F,\overline{G}_j\,\chi_m\big\rangle_{L^2(\widehat{H})}&=
\int_{\widehat{H}} F(\xi)G_j(\xi)\overline{\chi_m(\xi)} \,dm_{\widehat{H}}(\xi)= \sum_{n=0}^{r-1}
\int_{\mu^\perp_{n}+\Omega} F(\xi)G_j(\xi)\overline{( m,\xi )} \,dm_{\widehat{H}}(\xi)\\
&=\int_{\Omega} \sum_{n=0}^{r-1} F(\xi+\mu^\perp_{n})G_j(\xi+\mu^\perp_{n})\overline{(m,\xi )} \,dm_{\widehat{H}}(\xi)\,.
\end{split}
\]
For $G_{j}\in L^2(\widehat{H})$, $j=1,2,\ldots,s$, consider the associated $s\times r$ matrix 
\begin{equation}
\label{Gmatrix}
\mathbb{G}(\xi):=
\begin{bmatrix} G_1(\xi)& G_1(\xi+\mu_1^\bot)&\cdots&G_1(\xi+\mu_{r-1}^\bot)\\
G_2(\xi)& G_2(\xi+\mu_1^\bot)&\cdots&G_2(\xi+\mu_{r-1}^\bot)\\
\vdots&\vdots&&\vdots\\
G_s(\xi)& G_s(\xi+\mu_1^\bot)&\cdots&G_s(\xi+\mu_{r-1}^\bot)
\end{bmatrix}=
\bigg[G_j\Big(\xi+\mu_k^\bot\Big)\bigg]_{\substack{j=1,2,\ldots,s \\ k=0,1,\ldots, r-1}}
\end{equation}
where $\xi$ is a element of $\Omega$ and its related constants
\[
\alpha_{\mathbb{G}}:=\einf_{\xi \in \Omega}\lambda_{\min}[\mathbb{G}^*(\xi)\mathbb{G}(\xi)]\quad \text{and} \quad
\beta_{\mathbb{G}}:=\esup_{\xi \in \Omega}\lambda_{\max}[\mathbb{G}^*(\xi)\mathbb{G}(\xi)]\,,
\]
where $\mathbb{G}^*(\xi)$ denotes the transpose conjugate of the matrix
$\mathbb{G}(\xi)$, and $\lambda_{\min}$ (respectively $\lambda_{\max}$) the
smallest (respectively the largest) eigenvalue of the positive semidefinite matrix
$\mathbb{G}^*(\xi)\mathbb{G}(\xi)$. Observe that
$0 \leq \alpha_{\mathbb{G}} \leq \beta_{\mathbb{G}} \leq \infty$.

\begin{prop}
\label{caracterizacion}
For $G_{j}\in L^2(\widehat{H})$, $j=1,2,\ldots,s$, consider the associated matrix $\mathbb{G}(\xi)$ given in \eqref{Gmatrix}. Then
\begin{enumerate}[(a)]
\item $\{\overline{G}_j\,\chi_m \}_{m\in M;\,j=1,2,\ldots,s}$ is a complete system for
$L^2(\widehat{H})$ if and only if the rank of $\mathbb{G}(w)$ is $r$,\,  a.e. in $\Omega$.
\item $\{\overline{G}_j\,\chi_m\}_{m\in M;\,j=1,2,\ldots,s}$ is a Bessel sequence for  $L^2(\widehat{H})$ if and only if $\beta_{\mathbb{G}}<\infty$. In this case, the optimal Bessel bound is $\beta_{\mathbb{G}}/r$.
\item $\{\overline{G}_j\,\chi_m\}_{m\in M;\,j=1,2,\ldots,s}$  is a frame for $L^2(\widehat{H})$ if and only if $0<\alpha_{\mathbb{G}}\le
\beta_{\mathbb{G}}<\infty$. In this case, the optimal frame bounds are $\alpha_{\mathbb{G}}/r$ and  $\beta_{\mathbb{G}}/r$.
\item  $\{\overline{G}_j\,\chi_m\}_{m\in M;\,j=1,2,\ldots,s}$  is a Riesz basis for $L^2(\widehat{H})$ if and only if it is a frame and $s=r$.
\end{enumerate}
\end{prop}
\begin{proof}
Let $F\in L^2(\widehat{H})$; from expression \eqref{ole} we get
\begin{equation}\label{a}
\sum_{j=1}^s \sum_{m\in M} \big| \big\langle F,\overline{G}_j\,\chi_m\big\rangle_{L^2(\widehat{H})}\big|^2=
\sum_{j=1}^s \sum_{m\in M} \big| (\mathbf{1}_\Omega\mathbf{G}_j^\top \mathbf{F})^{\wedge}(m) \big|^2\,.
\end{equation}
Let $L_s^2(\Omega)$ denotes the space of the functions
$\mathbf{H}=[H_1, H_2, \ldots, H_s]^\top$  such that
\[
\|\mathbf{H}\|^2_{L_s^2(\Omega)}:= \int_\Omega|\mathbf{H}(\xi)|^2 \,dm_{\widehat{H}}(\xi)=\sum_{j=1}^s
\|H_j\|_{L^2(\Omega)}^2 <\infty\,.
\] 
By using \eqref{a} and the generalized Parseval identity
given in \cite[Lemma 2.3]{ole:15}, we obtain that if $\mathbb{G}(\xi)\,\mathbf{F}(\xi) \in L_s^2(\Omega)$ then
\begin{equation}
\label{representation}
\sum_{j=1}^s \sum_{m\in M} \big| \big\langle F,\overline{G}_j\,\chi_m\big\rangle_{L^2(\widehat{H})}\big|^2=
\sum_{j=1}^s \dfrac{1}{r} \int_\Omega |\mathbf{G}_j^\top(\xi) \mathbf{F}(\xi)|^2 dm_{\widehat{H}}(\xi)= 
\dfrac{1}{r} \, \|\mathbb{G}\,\mathbf{F}\|^2_{L_s^2(\Omega)}\,.
\end{equation}
From representation \eqref{representation}  one can deduce $(a)$, $(b)$ and $(c)$ in the same way that in the proof of Lemma 1 in \cite{garcia:09}. There,
the lemma was proved, for the case $H=\mathbb{Z}^d$ and $M=\{An: n\in \mathbb{Z}^d\}$ ($A$ a $d \times d$ matrix with integer entries and positive determinant), from  representation \eqref{representation} which is Equation (14) in \cite{garcia:09}. 

To prove $(d)$ assume first that $\{\overline{G}_j\,\chi_m\}_{m\in M;\,j=1,2,\ldots,s}$ is a frame and $s=r$. From $(c)$, and having in mind that in this case $\mathbb{G}$ is a square matrix, we deduce that $\mathbb{G}$ has a inverse matrix with entries in $L^2(\widehat{H})$. Let $\mathbf{F}_{1}(\xi),\ldots,\mathbf{F}_{s}(\xi)$ denote the columns of  $\mathbb{G}^{-1}(\xi)$
and let $F_j\in L^2(\widehat{H})$ be such that $\mathbf{F}_{j}(\xi)=\big[ F_{j}(\xi+\mu_0^\bot),  F_{j}(\xi+\mu_1^\bot),\ldots, F_{j}(\xi+\mu_{r-1}^\bot) \big]^\top$, $\xi\in \Omega$.  Applying \eqref{ole} for the case $G_j=F=1$ we obtain that $\delta_m=\int_{\widehat{H}} \overline{\chi}_m(\xi)\,dm_{\widehat{H}}(\xi) =\int_\Omega r\, \overline{\chi}_m(\xi)\,dm_{\widehat{H}}(\xi)$. Using  \eqref{ole} again, we obtain
\[
\begin{split}
&\big\langle rF_{j'}\,\chi_{m'},\overline{G}_j\,\chi_m\big\rangle_{L^2(\widehat{H})}=
\big\langle rF_{j'},\overline{G}_j\,\chi_{m-m'}\big\rangle_{L^2(\widehat{H})}=\int_{\Omega} r\mathbf{G}_j^\top(\xi)\,\mathbf{F}_{j'}(\xi)\,
\overline{\chi}_{m-m'}(\xi)\, dm_{\widehat{H}}(\xi) \\
&= \delta_{j,j'}\int_{\Omega} r\,\overline{\chi}_{m-m'}(\xi)\, dm_{\widehat{H}}(\xi) =  \delta_{j,j'} 
\delta_{m,m'},\quad m,m'\in M, \,\,\, j,j'=1,2, \ldots,s\,.
\end{split} 
\]Hence, $\{\overline{G}_j\,\chi_m\}_{m\in M;\,j=1,2, \ldots,s}$ has a biorthogonal sequence and, as a consequence, it is a Riesz basis for $L^2(\widehat{H})$ (see \cite[Theorem 6.1.1]{ole:03}).

\medskip

To prove the reciprocal, assume that $\{\overline{G}_j\,\chi_m\}_{m\in M;\,j=1,2,\ldots,s}$ is a Riesz basis. Then it is a frame and has a biorthogonal sequence denoted by $\{F_{j,m}\}_{m\in M;\,j=1,2, \ldots,s}$. For $j=1, 2, \ldots, s$, set $\mathbf{F}_{j}(\xi)=\big[F_{j,0}(\xi+\mu_0^\bot), F_{j,0}(\xi+\mu_1^\bot), \dots, F_{j,0}(\xi+\mu_{r-1}^\bot)\big]^\top$. Using  \eqref{ole}, we obtain
\[
\int_{\Omega} \mathbf{G}_j^\top(\xi)\mathbf{F}_{j'}(\xi)
\overline{(m,\xi)}\,dm_{\widehat{H}}(\xi)= \langle F_{j',0}\, ,\overline{G}_j\,\chi_m\big\rangle_{L^2(\widehat{H})}= \delta_{j,j'} \delta_{m}\,, 
\]
for $m\in M$ and $j,j'=1,2,\ldots,s$. Since $\{\chi_m\}_{m\in M}$ is an orthogonal basis for $L^2(\Omega)$, we get
\[
 \mathbf{G}_j^\top(\xi)\mathbf{F}_{j'}(\xi)= \delta_{j,j'},\quad \text{a.e.}\,\, \xi\in \Omega\,,\quad  j,j'=1,2,\ldots,s\,.
\]
Hence matrix $\mathbb{G}(\xi)$ has a right-inverse a.e $\xi\in \Omega$. Since it is a frame, from $(c)$, we have $\alpha_{\mathbb{G}}>0$, and thus it also has a left-inverse a.e $\xi\in \Omega$ (for example,  $[\mathbb{G}^*(\xi)\mathbb{G}(\xi)]^{-1}\mathbb{G}^*(\xi)$). As a consequence, $\mathbb{G}(\xi)$ is a square matrix and $s=r$.
\end{proof}

\medskip

It is worth to mention that, in Ref.~\cite{ole:15}, Christensen and Goh have analyzed sequences of the type
$\big\{ \overline{G}_j\, \chi_m \big\}_{m\in M_j;\, j\in J}$ in $L^2(\widehat{H})$ where $H$ is a LCA group, $J$ is a countable set of index, and
$M_j$ is a uniform lattice that depends on $j$. Thus, Proposition~\ref{caracterizacion} corresponds to the simpler case in which $H$ is discrete, $J$  is finite and $M_j=M$, but our result provides more information on the sequence $\big\{ \overline{G}_j\, \chi_m \big\}_{m\in M_j;\, j\in J}$ than those included in \cite{ole:15}; this extra information will be needed to deduce our sampling results in next section.

\medskip

Notice that as a consequence of the equivalence between the spectral and the Frobenius norms (see, for instance, \cite{horn:99}), it follows that  
$\beta_{\mathbb{G}}<\infty$ if and only if the functions $G_{j}\in L^\infty(\widehat{H})$, $j=1, 2, \ldots,s$. Under this circumstance we have:

\begin{prop}
\label{caracterizacionframe} 
Assume that the functions $G_{j}$, $j=1, 2, \ldots,s$, belong to $L^\infty(\widehat{H})$. Then,  the sequence $\big\{\overline{G}_j\,\chi_m\big\}_{m\in M;\,j=1,2,\ldots,s}$  is a frame for $L^2(\widehat{H})$ if and only if 
\[
\einf_{\xi \in \Omega} \det \big[\mathbb{G}^*(\xi)\mathbb{G}(\xi)\big] > 0.
\]
Moreover, provided that  the functions $G_{j}$, $j=1, 2,\ldots,s$, are continuous on $\widehat{H}$, the sequence $\{\overline{G}_j\,\chi_m\}_{m\in M;\,j=1,2,\ldots,s}$  is a frame for $L^2(\widehat{H})$ if and only if
\[
\operatorname{rank} \mathbb{G}(\xi)=r\,\,\, \text{for all \,} \xi\in \widehat{H}.
\]
\end{prop}
\begin{proof} Since $\det \big[\mathbb{G}^*(\xi)\mathbb{G}(\xi)\big]$  is the product of its eigenvalues, we get 
\[
\alpha_{\mathbb{G}}^r \le \det [\mathbb{G}^*(\xi)\mathbb{G}(\xi)] \le \beta_{\mathbb{G}}^{r-1} \lambda_{\min}[\mathbb{G}^*(\xi)\mathbb{G}(\xi)]\,.
\]
Hence, $\alpha_{\mathbb{G}}>0$ if and only if $\einf_{\xi \in \Omega} \det [\mathbb{G}^*(\xi)\mathbb{G}(\xi)] > 0$. Whenever the functions $G_{j}$ are continuous on $\widehat{H}$, this condition is equivalent to $\det \big[\mathbb{G}^*(\xi)\mathbb{G}(\xi)\big]\neq 0$ for all $\xi\in \widehat{H}$.
\end{proof}

Concerning the appropriate dual frames of the sequence $\{\overline{G}_j\,\chi_m\}_{m\in M;\,j=1,2,\ldots,s}$, having in mind Proposition~\ref{caracterizacion}(b) and Proposition~3.13 in \cite{ole:15} we have the following result:
\begin{prop}
\label{proposicionaes} 
Assume that the functions $G_{j}, H_j\in L^\infty(\widehat{H})$, $j=1, 2, \ldots,s$, satisfy
\begin{equation}
\label{aes}
\mathbf{H}(\xi)^\top\,\mathbb{G}(\xi)=[1, 0, \dots, 0]\,,\quad \text{a.e.}\,\,\, \xi\in \widehat{H}\,,
\end{equation}
where $\mathbf{H}(\xi)=[H_1(\xi), H_2(\xi), \ldots, H_s(\xi)]^\top$. Then, the sequences  $\{\overline{G}_j\,\chi_m\}_{m\in M;\,j=1,2,\ldots,s}$  and $\{rH_j\,\chi_m\}_{m\in M;\,j=1,2,\ldots,s}$ form a pair of dual frames for $L^2(\widehat{H})$.
\end{prop}

 Concerning the existence of a vector $\mathbf{H}(\xi)=[H_1(\xi), H_2(\xi), \ldots, H_s(\xi)]^\top$ with entries in $L^\infty(\widehat{H})$ and satisfying \eqref{aes} we prove the following result:
\begin{prop}
\label{proposicionany} 
Assume that the functions $G_{j}\in L^\infty(\widehat{H})$, $j=1,\ldots,s$, and also that $\einf_{\xi \in \Omega} \det \,[\mathbb{G}^*(\xi)\mathbb{G}(\xi)] > 0$.
Then all the possible vectors $\mathbf{H}(\xi)$ satisfying \eqref{aes} with entries in $L^\infty(\widehat{H})$ are given by the first row of the $r\times s$ matrices
\begin{equation}
\label{many}
\mathbb{G}^\dag(\xi)+\mathbb{U}(\xi)\big[\mathbb{I}_s-\mathbb{G}(\xi)\mathbb{G}^\dag(\xi)\big]\,
\end{equation}
where $\mathbb{G}^\dag(\xi)=\big[\mathbb{G}^*(\xi)\,\mathbb{G}(\xi)\big]^{-1}\,\mathbb{G}^*(\xi)$ denotes the Moore-Penrose pseudo-inverse of 
$\mathbb{G}(w)$, $\mathbb{I}_s$ is the identity matrix, and $\mathbb{U}(\xi)$ denotes any $r\times s$ matrix with entries in $L^\infty(\widehat{H})$.
\end{prop}
\begin{proof} It is easy to check that the first row $\mathbf{H}(\xi)^\top$ of $\mathbb{G}^\dag(\xi)+\mathbb{U}(\xi)\big[\mathbb{I}_s-\mathbb{G}(\xi)\mathbb{G}^\dag(\xi)\big]$ satisfies \eqref{aes} having entries in $L^\infty(\widehat{H})$. Reciprocally, for any $\mu^\bot_ i,\mu^\bot_ j$ in the subgroup $M^\bot$, there exist $l$ such that $\mu^\bot_l=\mu^\bot_j-\mu^\bot_i$. From  \eqref{aes} we have
$\sum_{j=1}^s H_j(\xi)\, G_j(\xi+\mu^\perp_l)= \delta_l$, a.e. $\xi\in \widehat{H}$. Thus $\sum_{j=1}^s H_j(\xi+\mu^\perp_ i)\, G_j(\xi+\mu^\perp_ j)= \delta_{i,j}$, a.e. $\xi\in \widehat{H}$. Then, the matrix $\mathbb{H}^\top(\xi)$, where $\mathbb{H}(\xi)$ is defined as in \eqref{Gmatrix}, is a left-inverse of $\mathbb{G}(\xi)$, and it can be written as in \eqref{many} taking  $\mathbb{U}(\xi)=\mathbb{H}^\top(\xi)$ which has entries  in $L^\infty(\widehat{H})$.
\end{proof}
\section{Generalized sampling formulas in $\mathcal{A}_a$}
\label{section3}
Having in mind expression \eqref{samples} for the samples, as a consequence of Proposition~\ref{caracterizacion}(a) the sequence of samples $\{\mathcal{L}_jx(m)\}_{m\in M; \,j=1,\ldots,s}$ determines uniquely each $x \in \mathcal{A}_{a}$  if and only if  $\text{rank}\,\mathbb{G}(\xi)=r$, a.e. $\xi \in \Omega$. 

Furthermore, from Proposition~\ref{caracterizacionframe} stable recovery is possible provided that the functions $G_{j}\in L^\infty(\widehat{H})$, $j=1,2,\ldots,s$, and $\einf_{\xi \in \Omega} \det \,[\mathbb{G}^*(\xi)\mathbb{G}(\xi)] > 0$. Next we prove that each dual frame of $\{\overline{G}_j\,\chi_m\}_{m\in M;\,j=1,2,\ldots,s}$ provided by Proposition~\ref{proposicionaes} gives a sampling formula performing this recovery. 

Let assume that  $G_{j}, H_j\in L^\infty(\widehat{H})$, $j=1, 2, \ldots,s$, such that \eqref{aes} holds.
By using Proposition~\ref{proposicionaes}, the sequences $\{\overline{G}_j\,\chi_m\}_{m\in M;\,j=1,2,\ldots,s}$  and $\{r H_j\,\chi_m\}_{m\in M;\,j=1,2,\ldots,s}$ are dual frames for $L^2(\widehat{H})$. Hence, for any $F\in L^2(\widehat{H})$
\[
 F(\gamma)=r\sum_{m\in M}\sum_{j=1}^s \big\langle F,\overline{G}_j\,\chi_m\big\rangle_{L^2(\widehat{H})} \, H_j(\gamma)\,\chi_m(\gamma) \quad \text{in $L^2(\widehat{H})$}.
\]
Since  $\mathcal{L}_jx(m)=\big\langle F, \overline{G}_j\,\chi_m\big\rangle_{L^2(\widehat{H})}$, where $x=\mathcal{T}_{H,a}F$, we get
\[
  F(\gamma)=r\sum_{m\in M}\sum_{j=1}^s \mathcal{L}_jx(m) \, H_j(\gamma)\,\chi_m(\gamma) \quad \text{in $L^2(\widehat{H})$}.
\]
 Finally, the isomorphism $\mathcal{T}_{H,a}$ and the shifting property \eqref{sp} give, for any $x\in \mathcal{A}_a$, the sampling formula
\begin{equation}
\label{samplingFormula}
\begin{split}
 x&=\sum_{m\in M}\sum_{j=1}^s\mathcal{L}_jx(m)\mathcal{T}_{H,a}\big[r H_j(\cdot)\chi_m(\cdot)\big]
   =\sum_{m\in M}\sum_{j=1}^s\mathcal{L}_jx(m)U(m)\mathcal{T}_{H,a}[r H_j(\cdot)]\\
   &=\sum_{m\in M}\sum_{j=1}^s\mathcal{L}_jx(m)U(m)c_{j,h}\,,
\end{split}
\end{equation}
where $c_{j, h}:=\mathcal{T}_{H,a}(rH_j)\in \mathcal{A}_a$, $j=1,2,\dots, s$. Besides, 
$\big\{U(m) c_{j,h}  \big\}_{m\in M;\,j=1,2,\ldots,s}$ is a frame for $\mathcal{A}_a$.  

Note that, in the oversampling setting, i.e., whenever $s>r$, according to \eqref{many}  there exist infinitely many sampling functions $U(m)c_{j,h}$, $j=1, 2, \ldots,s$, for which sampling formula \eqref{samplingFormula} holds. This remarkable fact could be used to obtain sampling functions with prescribed properties.

Collecting all the pieces that we have obtained until now we prove the following result:
\begin{teo}
\label{regular}
Let $b_j\in \mathcal{H}$ and let $\mathcal{L}_j$ be its associated $U$-system for $j=1,2,\dots,s$. Assume that the functions $G_j$, $j=1,2,\dots,s$, given in \eqref{g_j} belong to $L^\infty(\widehat{H})$; or equivalently, that $\beta_{\mathbb{G}}<\infty$ for the associated $s\times r$ matrix $\mathbb{G}(\xi)$ given in \eqref{Gmatrix}. The following statements are equivalent:
\begin{enumerate}[(a)]
\item $\displaystyle{\einf_{\xi \in \Omega} \det \,[\mathbb{G}^*(\xi)\mathbb{G}(\xi)] > 0}$.
\item $\alpha_{\mathbb{G}}>0$.
\item There exists a vector $\big[H_1(\xi), H_2(\xi), \dots, H_s(\xi)\big]$ with entries in $L^\infty(\widehat{H})$ satisfying
\begin{equation}
\label{Hcondition}
\big[H_1(\xi), H_2(\xi), \dots, H_s(\xi)\big]\, \mathbb{G}(\xi)=[1, 0, \dots, 0] \quad \text{a.e.  $\xi$ in $\widehat{H}$}\,.
\end{equation}
\item There exist $c_j\in \mathcal{A}_a$, $j=1,2, \dots, s$, such that the sequence
$\big\{U(m)c_j\big\}_{m\in \mathbb{Z};\,j=1,2,\dots s}$ is a frame for $\mathcal{A}_a$, and for any $x\in \mathcal{A}_a$ the expansion
\begin{equation}
\label{framesampling}
x=\sum_{j=1}^s \sum_{m \in M}\mathcal{L}_j x(m)\,U(m)c_j \quad \text{in $\mathcal{H}$}\,,
\end{equation}
holds. 
\item There exists a frame $\big\{C_{j,m}\big\}_{m\in M;\,j=1,2,\dots s}$ for $\mathcal{A}_a$ such that, for each $x\in \mathcal{A}_a$ the expansion
\[
x=\sum_{j=1}^s \sum_{m\in M}\mathcal{L}_{j} x(m)\,C_{j,m} \quad \text{in $\mathcal{H}$}\,,
\]
holds.
\end{enumerate}
In case the equivalent conditions are satisfied, for the elements $c_j$ in $(d)$ we have $c_{j}=\mathcal{T}_{H,a}(rH_j)$, for some functions $H_j$ in $L^\infty(\widehat{H})$, $j=1,2,\dots, s$, and satisfying \eqref{Hcondition}.
\end{teo}
\begin{proof} Propositions~\ref{caracterizacionframe} and \ref{caracterizacion} $(c)$ prove that conditions $(a)$ and $(b)$ are equivalents.
Proposition~\ref{proposicionany} proves that condition $(a)$ implies condition $(c)$. We have proved above  that condition $(c)$ implies condition  $(d)$. Obviously, condition $(d)$ implies condition $(e)$. As a consequence, we only need  to prove that condition $(e)$ implies condition $(a)$. Applying the isomorphism $\mathcal{T}_{H,a}^{-1}$ to the expansion in $(e)$, and taking into account \eqref{samples} we obtain
\[
\begin{split}
F& =\mathcal{T}_{H,a}^{-1} x =\sum_{j=1}^s \sum_{m\in M}\mathcal{L}_{j} x(m)\,\mathcal{T}_{H,a}^{-1}\big(C_{j,m}\big) \\
& =\sum_{j=1}^s \sum_{m\in M} \big\langle F, \overline{G}_j\, \chi_m \big\rangle_{L^2(\widehat{H})}\,\mathcal{T}_{H,a}^{-1}\big(C_{j,m}\big)\quad \text{in $L^2(\widehat{H})$}\,,
\end{split}
\]
where the sequence $\big\{\mathcal{T}_{H,a}^{-1}\big(C_{j,m}\big) \big\}_{m\in M;\,  j=1,2,\dots s}$ is a frame for $L^2(\widehat{H})$. The sequence $\big\{ \overline{G}_j\, \chi_m \big\}_{m\in M;\,  j=1,2,\dots s}$ is a Bessel sequence in $L^2(\widehat{H})$ since $\beta_{\mathbb{G}}<\infty$, and it satisfies the above expansion in $L^2(\widehat{H})$. As a consequence, according to \cite[Lemma 5.6.2]{ole:03} the sequences  $\big\{ \overline{G}_j\, \chi_m \big\}_{m\in M;\,  j=1,2,\dots s}$ and $\big\{\mathcal{T}_{H,a}^{-1}(C_{j,m}) \big\}_{m\in M;\,  j=1,2,\dots s}$ form a pair of dual frames in $L^2(\widehat{H})$. In particular, by using Proposition~\ref{caracterizacionframe} we obtain that condition~$(a)$ holds, which concludes the proof.
\end{proof}

Whenever $r$ equals the number of $U$-systems $s$ we are in the presence of Riesz bases, and there exists a unique sampling expansion in Theorem \ref{regular}:

\begin{cor}
Let $b_j\in \mathcal{H}$ for $j=1,2,\dots,r$, i.e., $r=s$ in Theorem\,\ref{regular}.  Let $\mathcal{L}_j$ be its associated $U$-system for $j=1,2,\dots, r$. Consider the associated $r\times r$ matrix $\mathbb{G}(\xi)$ given in \eqref{Gmatrix}. The following statements are equivalent:
\begin{enumerate}[(a)]
\item $0<\alpha_{\mathbb{G}} \le \beta_{\mathbb{G}}<\infty$.
\item There exist $r$ unique elements $c_j \in \mathcal{A}_a$, $j=1,2,\dots, r$, such that the sequence
$\big\{U(m)c_j\big\}_{m\in \mathbb{Z};\,j=1,2,\dots r}$ is a Riesz basis for $\mathcal{A}_a$, and for any $x\in \mathcal{A}_a$ the expansion
\begin{equation}
\label{rieszsampling}
x=\sum_{j=1}^r \sum_{m \in M}\mathcal{L}_j x(m)\,U(m)c_j \quad \text{in $\mathcal{H}$}\,,
\end{equation}
holds.
\end{enumerate}
In case the equivalent conditions are satisfied, the interpolation property $\mathcal{L}_{j'} c_j(m)=\delta_{j,j'}\, \delta_{m,0}$, where $m\in M$ and $j, j'=1,2,\dots, r$, holds.
\end{cor}
\begin{proof}
Assume that $0<\alpha_{\mathbb{G}} \le \beta_{\mathbb{G}}<\infty$; since $\mathbb{G}(\xi)$ is a square matrix, this implies that
\[
\einf_{\xi\in\Omega}|\det \mathbb{G}(\xi)|>0\,.
\]
Therefore, the first row of $\mathbb{G}^{-1}(\xi)$ gives the unique solution $[H_1(\xi), H_2(\xi), \dots, H_r(\xi)]$ of \eqref{aes}
with $H_j\in L^\infty(\widehat{H})$ for $j=1, 2,\ldots, r$. Indeed, other vector  $[\widetilde{H}_1(\xi), \widetilde{H}_2(\xi), \dots, \widetilde{H}_r(\xi)]$ satisfying \eqref{aes}
would give, by defining $\tilde{\mathbb{H}}(\xi)$ as in \eqref{Gmatrix}, another inverse of $\mathbb{G}(\xi)$.

Due to  Theorem \ref{regular}, the sequence $\{U(m) c_j\}_{m\in M;\,j=1,2,\ldots,r}$,
with $c_j=\mathcal{T}_{H,a}(rH_j)$,  satisfies the sampling
formula \eqref{rieszsampling}. Moreover, the sequence $\{rH_j(w)\,\chi_m\}_{m\in M;\,j=1,2,\ldots,r}=\{\mathcal{T}_{H,a}^{-1}\big(U(m)c_j\big)\}_{m\in M;\,j=1,2,\ldots,r}$ is a frame for $L^2(\widehat{H})$.  Since $r=s$,  according to Proposition~\ref{caracterizacion}(d), it is a Riesz basis. Hence,
$\{U(m)c_j\}_{m\in M;\,j=1,2,\ldots,r}$ is a Riesz basis for $\mathcal{A}_a$. It is straightforward to prove the uniqueness of $c_j$, $j=1,2,\dots,r$. 

For the converse result, the isomorphism $\mathcal{T}_{H,a}^{-1}$ in \eqref{rieszsampling} gives the Riesz basis expansion
\[
F=\sum_{j=1}^r \sum_{m \in M}\big\langle F,\overline{G}_j\,\chi_m\big\rangle_{L^2(\widehat{H})}\,\mathcal{T}_{H,a}^{-1}\big[U(m)c_j\big] \quad \text{in $L^2(\widehat{H})$}\,.
\]
Moreover, due to the uniqueness of the coefficients in a Riesz basis, the sequence $\big\{\overline{G}_j\,\chi_m\big\}_{m\in H;\,j=1,2,\ldots,r}$ must be the dual  Riesz basis of $\big\{\mathcal{T}_{H,a}^{-1}[U(m)c_j]\big\}_{m\in H;\,j=1,2,\ldots,r}$. According to Proposition~\ref{caracterizacion}(d), condition $(a)$ holds.

From the uniqueness of the coefficients in a Riesz basis, we get the interpolatory condition
$(\mathcal{L}_{j'}c_j)(m)=\delta_{j,j'}\delta_{m,0}$ for $j,j'=1,2,\ldots,r$ and $m \in M$.
\end{proof}

\subsection*{Some illustrative examples}
The general sampling theory obtained in Section \ref{section3} englobes some previous work in the topic. Namely:
\begin{itemize}
\item The case $G=\mathbb{R}$, $H=\mathbb{Z}$ and $M=r\mathbb{Z}$ $(r\in \mathbb{N})$, with unitary representation $U(n)=U^n$, $n\in \mathbb{Z}$, where $U: \mathcal{H} \rightarrow \mathcal{H}$ is a unitary operator on an abstract Hilbert space $\mathcal{H}$, can be found in Refs. \cite{hector:13,hector:14,pohl:12}. In particular, whenever $U: f(t)\mapsto f(t-1)$ is the shift operator in $L^2(\mathbb{R})$, see \cite{garcia:06}.
\item The case $G={\mathbb{R}}^d$, $H={\mathbb{Z}}^d$ and $M=A{\mathbb{Z}}^d$, where $A$ denotes a $d\times d$ matrix with integer entries and positive determinant, with unitary representation in $L^2({\mathbb{R}}^d)$ given by $U(\alpha):f(t)\mapsto f(t-\alpha)$, 
$\alpha \in {\mathbb{Z}}^d$, can be found in \cite{garcia:09}. 
\item The case  $H=\mathbb{Z}_N$ and $M=r\mathbb{Z}_N$ ($r|N$), with unitary representation $U(n)=U^n$, $n\in \mathbb{Z}_N$, where $U: \mathcal{H} \rightarrow \mathcal{H}$ is a unitary operator on a Hilbert space $\mathcal{H}$ can be found in \cite{garcia:15}. In particular, it includes the case where $U$ is the cyclic shift in $\ell^2_N(\mathbb{Z})$; see also Ref.~\cite{garcia:16}. 
\item The case of the product group $H=\mathbb{Z} \times \mathbb{Z}_N$ and $M=r\mathbb{Z} \times \bar{r}\mathbb{Z}_N$ ($r\in \mathbb{N}$ and  $\bar{r}|N$), with unitary representation $U(n,m)=U^n \otimes V^m$, $(n, m)\in \mathbb{Z}\times\mathbb{Z}_N$, in the tensorial product $\mathcal{H}_1\otimes \mathcal{H}_2$, where $U: \mathcal{H}_1 \rightarrow \mathcal{H}_1$ and $V: \mathcal{H}_2 \rightarrow \mathcal{H}_2$ are unitary operators on the Hilbert spaces $\mathcal{H}_1$ and $\mathcal{H}_2$ respectively, can be found in \cite{garcia2:16}. See also the case  $H=\mathbb{Z}_N \times \mathbb{Z}_M$ and $M=r\mathbb{Z}_N \times \bar{r}\mathbb{Z}_M$ ($r|N$ and  $\bar{r}|M$).
\end{itemize}

\subsection*{Kluv\'anek's sampling theorem for $H$-shift-invariant subspaces in $L^2(G)$} 
\label{section4}
A similar technique to those used in Section \ref{section3} allows to derive a generalization of Kluv\'anek sampling theorem for $H$-shift-invariant subspaces of $L^2(G)$. In \cite{kluvanek:65} Kluv\'anek obtained a generalization of the classical Shannon sampling theorem in terms of abstract harmonic analysis (see also \cite{dodson:07}). On the other hand, classical Paley-Wiener spaces are, in particular, shift-invariant subspaces in $L^2(\mathbb{R})$; thus Walter proved in \cite{walter:92}, for the first time, a sampling theorem for shift-invariant subspaces of $L^2(\mathbb{R})$. Here we reproduce the same result  in the context of abstract harmonic analysis.
As before, $G$ denotes a LCA group and $H<G$ is a uniform lattice in $G$. We take $\mathcal{H}=L^2(G)$ and we consider the (left regular) unitary representation of $G$ given by $g\in G \mapsto L_g$, where $L_gx(g')=x(g'-g)$ for any $g'\in G$ and $x\in L^2(G)$. For a fixed $a\in L^2(G)$ we define the $H$-shift-invariant subspace $\mathcal{A}_a:=\overline{\operatorname{span}}\big\{ L_ha\,:\, h\in H\big\}$. In case the sequence $\{L_ha\}_{h\in H}$ is a Riesz sequence in $L^2(G)$ (see Ref.~\cite{cabrelli:10}), this subspace can be described as 
\[
\mathcal{A}_a=\Big\{\sum_{h\in H}\alpha_h\,a(g-h)\,:\,\{\alpha_h\}\in \ell^2(H)\Big\}\subset L^2(G)\,.
\]
Assuming that $a\in C(G)$ and $\sum_{h\in H}|a(g-h)|^2<\infty$ uniformly in $G$, one can easily derive that the subspace $\mathcal{A}_a$ constitutes a {\em reproducing kernel Hilbert space} (RKHS) of continuous functions on $G$. Thus, convergence in $L^2(G)$-sense implies pointwise convergence which is, in this case, uniform on $G$. Besides, each $x\in \mathcal{A}_a$ can be described as the pointwise sum $x(g)=\sum_{h\in H} \alpha_h\,a(g-h)$, \,$g\in G$. 

Following Ref.~\cite{sun:99} it can be proved, by using the Banach-Steinhaus theorem, that $a\in C(G)$ and $\sum_{h\in H}|a(g-h)|^2<\infty$ uniformly in $G$ is also a necessary condition in order to be $\mathcal{A}_a$ a $H$-shift-invariant space of continuous functions on $G$.

The aim in this section  is to recover any $x\in \mathcal{A}_a$ from the sequence of its samples taken at an orbit of the subgroup $H$, i.e., $\{x(g_0+m)\}_{m\in H}$ where $g_0$ denotes a fixed element in $G$. In this context, the isomorphism $\mathcal{T}_{H,a}$ in \eqref{iso} reads:
\[
\mathcal{T}_{H,a}: L^2(\widehat{H}) \ni F=\sum_{h\in H}\alpha_h\,\chi_h  \longmapsto x(g)=\sum_{h\in H}\alpha_h\,a(g-h)\in \mathcal{A}_a\,,
\]
which satisfies the shifting property $\mathcal{T}_{H,a}(F\chi_m)=L_m(\mathcal{T}_{H,a}F)$, for $F\in L^2(\widehat{H})$ and $m\in H$. We obtain the following expression for the samples of $x=\mathcal{T}_{H,a} F\in \mathcal{A}_a$: 
\[
\begin{split}
x(g_0+m)&=\sum_{h\in H} \langle F, \chi_h\rangle\,a(g_0+m-h)=\big\langle F, \sum_{h\in H}\overline{a(g_0+m-h)}\,\chi_h \big\rangle_{L^2(\widehat{H})}\\
&=\big\langle F, \sum_{h'\in H}\overline{a(g_0+h')}\,\chi_{m-h'} \big\rangle_{L^2(\widehat{H})}=\big\langle F, \big(\sum_{h'\in H}\overline{a(g_0+h')\,\chi_{h'}}\,\big)\chi_{m} \big\rangle_{L^2(\widehat{H})}\\
&= \langle F, \overline{G}_{g_0}\,\chi_m\rangle_{L^2(\widehat{H})}\,,
\end{split}
\]
where $G_{g_0}$ denotes the function
\begin{equation}
\label{g0}
G_{g_0}:=\sum_{h\in H}a(g_0+h)\,\chi_h \in L^2(\widehat{H})\,.
\end{equation} 
Having in mind Proposition~\ref{caracterizacion}(d), the sequence $\big\{ \overline{G}_{g_0}\,\chi_m\big\}_{m\in H}$ is a Riesz basis for $L^2(\widehat{H})$ if and only if
\begin{equation}
\label{sriesz}
0<\einf_{\gamma \in \widehat{H}} |G_{g_0}(\gamma)|\le \esup  _{\gamma \in \widehat{H}} |G_{g_0}(\gamma)| <+\infty\,.
\end{equation}
Note that taking $M=H$ in Proposition~\ref{caracterizacion} we have $\Omega=\widehat{H}$. Moreover, its dual Riesz basis is $\big\{ \big(1/G_{g_0}\big)\,\chi_m\big\}_{m\in H}$. Assuming \eqref{sriesz}, we obtain a sampling formula in $\mathcal{A}_a$ as follows: For each $x=\mathcal{T}_{H,a}^{-1}F\in \mathcal{A}_a$, we expand $F$ in $L^2(\widehat{H})$ as
\[
F=\sum_{m\in H}\langle F, \overline{G}_{g_0}\,\chi_m\rangle\, \frac{1}{G_{g_0}}\,\chi_m=\sum_{m\in H}x(g_0+m)\, \frac{1}{G_{g_0}}\,\chi_m\,.
\]
Then, applying the isomorphism $\mathcal{T}_{H,a}$ and the shifting property we get  in $\mathcal{A}_a$ the sampling formula 
\[
\begin{split}
x(g)&=\sum_{m\in H}x(g_0+m)\,\mathcal{T}_{H,a}\big( \frac{1}{G_{g_0}}\,\chi_m \big)(g)=\sum_{m\in H}x(g_0+m)\, \big(L_m S_{g_0}\big)(g) \\
&= \sum_{m\in H}x(g_0+m)\, S_{g_0}(g-m)\,, \quad g\in G\,,
\end{split}
\]
where the sampling function $S_{g_0}=\mathcal{T}_{H,a} \big(1/G_{g_0} \big)$ belongs to $\mathcal{A}_a$. The  convergence of the series is in the $L^2(G)$-sense, absolute (due to the unconditional convergence of a Riesz basis) and uniform on $G$. In fact, the following result holds:

\begin{teo}
Let $\mathcal{A}_a$ be a $H$-shift-invariant subspace of continuous functions on $G$ generated by a function $a\in L^2(G)$. 
For a fixed $g_0\in G$, consider the function $G_{g_0}\in L^2(\widehat{H})$ given in \eqref{g0}. The following statements are equivalent:
\begin{enumerate}[(i)]
\item $0<\einf_{\gamma \in \widehat{H}} |G_{g_0}(\gamma)|\le \esup  _{\gamma \in \widehat{H}} |G_{g_0}(\gamma)| <+\infty$.
\item  There exists a unique function $S_{g_0}\in\mathcal{A}_a$ such that the sequence $\big\{S_{g_0}(\cdot -m)\big\}_{m\in H}$ is a Riesz basis for $\mathcal{A}_a$, and the following sampling formula holds in $\mathcal{A}_a$:
\[
x(g)=\sum_{m\in H}x(g_0+m)\, S_{g_0}(g-m) \quad \text{in $L^2(G)$}\,.
\]  
\end{enumerate}
Moreover, in case of equivalent conditions the convergence of the above series is absolute and uniform on $G$.
\end{teo}
\begin{proof}
We have proved that condition $(i)$ implies condition $(ii)$. Reciprocally, applying the isomorphism $\mathcal{T}_{H,a}^{-1}$ in the sampling formula one gets the expansion
\[
F=\sum_{m\in H} \langle F, \overline{G}_{g_0}\,\chi_m\rangle_{L^2(\widehat{H})}\, \mathcal{T}_{H,a}^{-1}\big[S_{g_0}(\cdot-m)\big] \quad \text{in $L^2(\widehat{H})$}\,,
\]
where the sequence $\big\{\mathcal{T}_{H,a}^{-1}\big[S_{g_0}(\cdot-m)\big]\big\}_{m\in H}$ is a Riesz basis for  $L^2(\widehat{H})$. Due to the uniqueness of the coefficients in a Riesz basis, necessarily the sequence $\big\{\overline{G}_{g_0}\,\chi_m\big\}_{m\in H}$ is its dual  Riesz basis. According to Proposition~\ref{caracterizacion}(d), condition $(i)$ holds.
\end{proof}
Note that, in case the sequence $\{L_ha\}_{h\in H}$ is only a Bessel sequence for $L^2(G)$ condition $(i)$, in the above theorem, implies that the sampling formula holds for every $x=\sum_{h\in H}\alpha_h\,L_h a$ with $\{\alpha_h\}\in \ell^2(H)$.

\subsection*{A note on the $G$-jitter error}
\label{section5}
In the classical Paley-Wiener space $PW_\pi$, the time-jitter error arises when one samples at instants $t_n=n+\epsilon_n$ which differ from the Nyquist sampling instants $n\in \mathbb{Z}$ by $\epsilon_n$ (see, for instance, the classical Ref.~\cite{zayed:93}). Its counterpart here consists of taking the generalized samples at $m+\epsilon_m$ where $m\in M$ and $\epsilon_m \in G$ stands for what we shall call a $G$-jitter error. In what follows, we assume $G$ to be a non-compact continuous group, and $M< H< G$ two countably infinite uniform lattices in $G$.

Given an error sequence $\boldsymbol{\epsilon}:=\{\epsilon_{mj}\}_{m\in M;\, j=1,2,\ldots,s}$ in $G$, proceeding as in \eqref{samples}, the perturbed 
samples in $\big\{\mathcal{L}_jx(m+\epsilon_{mj})\big\}_{m\in M;\, j=1,2,\ldots,s}$ can be expressed as
\[
\mathcal{L}_jx(m+\epsilon_{mj})=\big\langle F,\overline{G}_{j,m} \,\chi_m\big\rangle_{L^2(\widehat{H})}\quad \text{for $m\in M$ and $j=1,2, \dots, s$}\,,
\]
where the function $G_{j,m}:=\sum_{h\in H} \mathcal{L}_ja(h+\epsilon_{mj})\,\chi_h$ belongs to $L^2(\widehat{H})$ for $m\in M$ and $j=1,2,\dots,s$. The idea consists of consider the new sequence $\big\{\overline{G}_{j,m} \,\chi_m\big\}_{m\in M}$ as a perturbation of the frame $\{\overline{G}_j\,\chi_m\}_{m\in M;\,j=1,2,\ldots,s}$  in $L^2(\widehat{H})$. For instance, if we prove that 
\[
\sum_{m\in M}\sum_{j=1}^s\big|\big\langle \overline{G}_{j,m}\, \chi_m-\overline{G}_j\, \chi_m  ,F\big\rangle_{L^2(\widehat{H})}\big|^2 \leq R \,\|F\|^2\quad \text{for each $F\in L^2(\widehat{H})$}
\]
with $R<\alpha_\mathbb{G}/r$, the optimal lower frame bound of the frame $\{\overline{G}_j\,\chi_m\}_{m\in M;\,j=1,2,\ldots,s}$, from 
Corollary 15.1.5 in \cite{ole:03} we obtain that the sequence $\big\{ \overline{G}_{j,m}\,\chi_m \big\}_{m\in M;\,  j=1,2,\dots, s}$ is also a frame for $L^2(\widehat{H})$.
In so doing, assume that the operator 
\[
\begin{array}[c]{ccll}
D_{\boldsymbol{\epsilon}}: & \ell^2(H) & \longrightarrow & \ell^2_s(M)\\
       & c=\{c_h\}_{h\in H} & \longmapsto & D_{\boldsymbol{\epsilon}}\,c:=\big(D_{\boldsymbol{\epsilon},1}\,c,\ldots,D_{\boldsymbol{\epsilon},s}\,c\big)\,,
\end{array}
\]
is well-defined, where, for $j=1, 2, \dots, s$,
\begin{equation}
\label{Dc}
D_{\boldsymbol{\epsilon},j}\,c:=\Big\{\sum_{h\in H}\big[\mathcal{L}_ja(m-h+\epsilon_{mj})-\mathcal{L}_ja(m-h)\big]c_h\Big\}_{m\in M}\,.
\end{equation}
The operator norm (it could be infinity) is defined as usual
\[
\|D_{\boldsymbol{\epsilon}}\|:=\sup_{c\in \ell^2(H)\setminus \{0\}} \frac{\|D_{\boldsymbol{\epsilon}}
\,c\,\|_{\ell^2_s(M)} }{\|c\,\|_{\ell^2(H)}}\,,
\]
where $\|D_\epsilon \,c\,\|^2_{\ell^2_s(M)}:=\sum_{j=1}^s
\|D_{\epsilon,j}\, c\,\|^2_{\ell^2(M)}$ for each $c\in \ell^2(H)$.

\begin{prop}
\label{irregular}
Assume that for the functions $G_j$, $j=1, 2, \dots, s$, given in \eqref{g_j} we have $0<\alpha_{\mathbb{G}}\le \beta_{\mathbb{G}}<\infty$. Let 
$\boldsymbol{\epsilon}:=\{\epsilon_{mj}\}_{m\in M;\,j=1, 2,\ldots,s}$ be an error sequence in $G$ satisfying
the inequality $\|D_{\boldsymbol{\epsilon}}\|^2<\alpha_\mathbb{G}/r$. Then, there 
exists a frame $\{C_{m,j}^{\boldsymbol{\epsilon}}\}_{m\in M;\,j=1,2,\ldots,s}$ for $\mathcal{A}_a$
such that, for any $x\in \mathcal{A}_a$, the sampling expansion
\begin{equation}
\label{firregular}
x=\sum_{j=1}^s \sum_{m\in M} \mathcal{L}_jx(m+\epsilon_{mj}) \, C_{m,j}^{\boldsymbol{\epsilon}} \quad \text{ in $\mathcal{A}_a$}\,,
\end{equation}
holds. Moreover, when $r=s$ the sequence $\{C_{m,j}^{\boldsymbol{\epsilon}}\}_{m\in M;\,j=1,2,\ldots,s}$
is a Riesz basis for $\mathcal{A}_a$, and the interpolation property 
$(\mathcal{L}_l \,C_{n,j}^{\boldsymbol{\epsilon}})(m+\epsilon_{mj})=\delta_{j,l}\,\delta_{n,m}$ holds.
\end{prop}

\begin{proof}
 The sequence 
$\big\{ \overline{G}_j\,\chi_m\big\}_{m\in M;\,  j=1,2,\dots, s}$
is a frame (a Riesz basis if $r=s$) for $L^2(\widehat{H})$ with optimal frame (Riesz) bounds
$\alpha_{\mathbb{G}}/r$ and $\beta_{\mathbb{G}}/r$. For any 
$F=\sum_{h \in H} \alpha_h \,\chi_h$ in $L^2(\widehat{H})$, a straightforward calculation gives
\begin{equation}
\begin{split}
&\sum_{m\in M}\sum_{j=1}^s\big|\big\langle \overline{G}_{j,m}\, \chi_m-\overline{G}_j\, \chi_m  ,F\big\rangle_{L^2(\widehat{H})}\big|^2\\
=&\sum_{m\in M}\sum_{j=1}^s\big|\big\langle \sum_{k\in H}\big( \overline{\mathcal{L}_ja(k+\epsilon_{mj})}-\overline{\mathcal{L}_ja(k)} \big)\chi_{-k}\chi_m  ,F\big\rangle_{L^2(\widehat{H})}\big|^2\\
=&\sum_{m\in M}\sum_{j=1}^s\big|\big\langle \sum_{h\in H}\big( \overline{\mathcal{L}_ja(m-h+\epsilon_{mj})}-\overline{\mathcal{L}_ja(m-h)} \big)\chi_h  ,F\big\rangle_{L^2(\widehat{H})}\big|^2\\
=&\sum_{m\in M}\sum_{j=1}^s\big|\sum_{h\in H}\big( \overline{\mathcal{L}_ja(m-h+\epsilon_{mj})}-\overline{\mathcal{L}_ja(m-h)} \big)\,\overline{\alpha}_h\big|^2\\
= &\sum_{j=1}^s \|D_{\epsilon,j}
\{\alpha_h\}_{h\in H}\|^2_{\ell^2(M)}\leq \|D_{\boldsymbol{\epsilon}}\|^2
\|\{\alpha_h\}_{h\in H}\|^2_{\ell^2(H)}= \|D_{\boldsymbol{\epsilon}}\|^2
\|F\|^2_{L^2(\widehat{H})}\,.
\end{split}
\end{equation}
By using Corollary 15.1.5 in \cite{ole:03} we obtain that the sequence 
$\big\{ \overline{G}_{j,m}\,\chi_m \big\}_{m\in M;\,  j=1,2,\dots, s}$ is a frame for $L^2(\widehat{H})$
(a Riesz basis if  $r=s$). Let $\{H_{m,j}^{\boldsymbol{\epsilon}}\}_{m\in M;\,j=1,2,\ldots,s}$ be, for instance,
its canonical dual frame. Hence, for any $F\in L^2(\widehat{H})$
\[
 F=\sum_{m\in M}\sum_{j=1}^s \big \langle F,
\overline{G}_{j,m}\, \chi_m\big\rangle_{L^2(\widehat{H})}\, H_{m,j}^{\boldsymbol{\epsilon}}
= \sum_{m\in M}\sum_{j=1}^s \mathcal{L}_j x(m+\epsilon_{mj})\,H_{m,j}^{\boldsymbol{\epsilon}}\quad \text{in $L^2(\widehat{H})$}\,.
\]
Applying the isomorphism $\mathcal{T}_{H,a}$, one gets \eqref{firregular}, where $C_{m,j}^{\boldsymbol{\epsilon}}:=\mathcal{T}_{H,a}\big(H_{m,j}^{\boldsymbol{\epsilon}}\big)$ for $m\in M$ and $j=1, 2, \dots, s$. Since $\mathcal{T}_{H,a}$ is an isomorphism between $L^2(\widehat{H})$ and 
$\mathcal{A}_a$, the sequence $\{C_{m,j}^{\boldsymbol{\epsilon}}\}_{m\in M;\,j=1,2,\ldots,s}$ is a frame for $\mathcal{A}_a$ (a Riesz basis if $r = s$). The interpolatory property in the case $r=s$ follows from the uniqueness of the coefficients with respect to a Riesz basis.
\end{proof}

Sampling formula \eqref{firregular} is useless from a practical point of view: it is impossible to determine the involved frame 
$\{C_{m,j}^{\boldsymbol{\epsilon}}\}_{m\in M;\,j=1,2,\ldots,s}$. As a consequence, in order to recover $x\in \mathcal{A}_a$ from the sequence of samples $\big\{(\mathcal{L}_j x)(m+\epsilon_{mj})\big\}_{m\in M;\,j=1,2,\ldots,s}$ a frame algorithm in 
$\ell^2(H)$ should be implemented.

\subsubsection*{On the existence of sequences $\boldsymbol{\epsilon}$ in $G$ satisfying $\|D_{\boldsymbol{\epsilon}}\|^2<\alpha_\mathbb{G}/r$} 

Given an error sequence $\boldsymbol{\epsilon}:=\{\epsilon_{mj}\}_{m\in M;\,j=1, 2,\ldots,s}$  in $G$, for  $j=1,2,\dots,s$, $m\in M$ and $h\in H$ we denote
\[
   q_{m,h}^{(j)}:=\mathcal{L}_ja(m-h+\epsilon_{mj})-\mathcal{L}_ja(m-h)\in \mathbb{C}\,.
\] 
Taking into account \eqref{Dc} and proceeding as in Ref~\cite{hector:14}, for any sequence $c=\{c_h\}_{h\in H}$ in $\ell^2(H)$  we have
\begin{equation}
\label{normoperator}
\|D_{\boldsymbol{\epsilon}}\|^2_{\ell^2_s(M)}=\sum_{j=1}^s \sum_{m\in M}
  \Big|\sum_{h\in H} q^{(j)}_{m,h}\ c_h \Big|^2 \le 
  \sum_{j=1}^s \sum_{l\in H} |c_l|^2 \sum_{(m,h) \in M\times H}|q^{(j)}_{m,l}\,q^{(j)}_{m,h}|\,.
\end{equation}
Let $d:G\times G\to [0,+\infty)$ denote an invariant metric in $G$ compatible with its topology. Since $H$ is a uniform lattice of $G$ the infimun of the set $\big\{d(h,h')\,:\, h,h'\in H, h\neq h'\big\}$ is greater than zero. Let  $0<\eta<\inf \{d(h,h')\,:\, h,h'\in H, h\neq h'\}$ and such that 
$\bar{B}_\eta(0)$, the closed ball with radius $\eta$ and centered at $0$, the identity in $G$, is a compact set of $G$. Recall that the so-called {\em oscillation} of the function $\mathcal{L}_j a$ in a  compact set $\mathcal{U}$ is given by (see, for instance, \cite{fuhr:07}):
\[
   \osc_{\mathcal{U}}(\mathcal{L}_j a)(h):=\max_{k\in\mathcal{U}} |\mathcal{L}_ja(h+k)-\mathcal{L}_ja(h)|\,.
\]
For $0\leq \delta\leq \eta$ we define the functions:
\[
M_{a,b_j}(\delta):=\sum_{h\in H} \osc_{\bar{B}_\delta(0)}(\mathcal{L}_j a)(h)\,,
\]
and 
\[
N_{a,b_j}(\delta):=\max_{l\in \Theta}\sum_{m\in M}\osc_{\bar{B}_\delta(0)}(\mathcal{L}_j a)(m-l)\,,
\]
where $\Theta$ is a section of the quotient group $H/M$. Notice that $N_{a,b_j}(\delta)\leq M_{a,b_j}(\delta)$.

If the continuous functions $G\ni g\mapsto \mathcal{L}_ja(g)$, $j=1,2,\dots,s$, satisfy a decay condition as $\mathcal{L}_ja(g)={\rm O}\big(d(0,g)^{-(1+\tau)}\big)$ when $d(0,g)\to\infty$ for some $\tau>0$ we may deduce that the functions $N_{a,b_j}$ and $M_{a,b_j}$ are continuous near to $0$. Taking into account \eqref{normoperator} the condition $\|D_{\boldsymbol{\epsilon}}\|^2<\alpha_\mathbb{G}/r$ in Proposition \ref{irregular} could be rephrased as the condition
\[
   \sum_{j=1}^s N_{a,b_j}(\delta) M_{a,b_j}(\delta)<\frac{\alpha_\mathbb{G}}{r}\,,
\]
for some small enough $\delta >0$, and $d(0,\epsilon_{mj})\le \delta$ for all $j=1, 2, \dots, s$ and $m\in M$.

\medskip

\noindent{\bf Acknowledgments:}
This work has been supported by the grant MTM2014-54692-P from the Spanish {\em Ministerio de Econom\'{\i}a y Competitividad (MINECO)}.


\end{document}